\documentclass[12pt]{amsart}
\usepackage[all]{xy}
\usepackage{amssymb}
\usepackage{amsthm}
\usepackage{amsmath}
\usepackage{amscd,enumitem}
\usepackage{verbatim}
\usepackage{eurosym}
\usepackage{graphicx}
\usepackage{dsfont}
\usepackage{stmaryrd}
\usepackage{color}
\usepackage{float}
\usepackage{longtable}
\usepackage{dcolumn}
\usepackage[mathscr]{eucal}
\usepackage[all]{xy}
\usepackage{hyperref}
\usepackage[title]{appendix}
\usepackage[usenames,dvipsnames]{xcolor}
\usepackage{bbm}
\usepackage[textheight=8.7in, textwidth=6.6in]{geometry}
\usepackage{multirow}
\usepackage{caption}
\theoremstyle{plain}
\newtheorem{theorem}{Theorem}[section]
\newtheorem{cor}[theorem]{Corollary}
\newtheorem{lemma}[theorem]{Lemma}
\newtheorem{proposition}[theorem]{Proposition}

\theoremstyle{definition}
\newtheorem{definition}[theorem]{Definition}

\theoremstyle{remark}
\newtheorem*{remark}{Remark}

\catcode`,\active

\catcode`\,12

\numberwithin{equation}{section}

\begin{document}
\title[Hypergeometric functions]
{Values of $p$-adic hypergeometric functions, and $p$-adic analogue of Kummer's linear identity}


%    author one information
 %\author[short version for running head]
{}
\author{Neelam Saikia}
\address{School of Basic Sciences, Discipline of Mathematics, Indian Institute of Technology Bhubaneswar, Argul-Jatni Rd, Kansapada, Odisha 752050, India}
\curraddr{}
\email{neelamsaikia@iitbbs.ac.in, nlmsaikia1@gmail.com}
\thanks{}

%    author two information

%    \subjclass is required.
\thanks{}
\begin{abstract} 
Let $p$ be an odd prime and $\mathbb{F}_p$ be the finite field with $p$ elements.
This paper focuses on the study of values of a generic family of hypergeometric functions in the $p$-adic setting which we denote by ${_{3n-1}G_{3n-1}}(p, t),$ where $n\geq1$ and $t\in\mathbb{F}_p$. These values are expressed in terms of numbers of zeros of certain polynomials over $\mathbb{F}_p$. These results lead to certain $p$-adic analogues of classical hypergeometric identities.
Namely, we obtain $p$-adic analogues of particular cases of a Gauss' theorem and a Kummer's theorem. Moreover, we examine the zeros of these functions. For instance, if $n$ is odd then we obtain zeros of ${_{3n-1}G_{3n-1}}(p, t)=0$ under certain condition on $t$. In contrast we show that if $n$ is even then the function ${_{3n-1}G_{3n-1}}(p, t)$ has no non-trivial zeros for any prime $p$. 

\end{abstract}

\maketitle

\section{Introduction and statement of results}
%The purpose of this paper is to present explicit formulas for values at $t\in\mathbb{F}_p$ of $p$-adic hypergeometric functions 
%introduced by D. McCarthy \cite{mccarthy-pacific}. These formulas yield $p$-adic analogues of classical hypergeometric series transformations, specially a very special case of Kummer's identity \cite[p. 4 eq. (1)]{bailey}.  McCarthy \cite{mccarthy-pacific} raised these questions of finding transformation identities in the $p$-adic setting that are analogous to classical hypergeometric identities.
%%%%%%%%%%%%%%%%%%%%%%%%%%%%%%%%%%%%%%%%%%%%%%%%%%%%%%%%%5
\noindent Let $a$ be a complex number and $k$ be a nonnegative integer. Then the rising factorial $(a)_k$ is defined as  
$(a)_0:=1$ for $k=0$ and $(a)_k:=a(a+1)(a+2)\cdots(a+k-1)$ for $k>0.$ Then for $a_i,b_i,\lambda\in\mathbb{C}$ with $b_i\neq0,-1,-2\ldots$, the ${_{r+1}F_r}$ classical hypergeometric series is defined by
\vspace{-3mm}
\begin{align*}
{_{r+1}}F_{r}\left(\begin{array}{cccc}
                   a_1, & a_2, & \ldots, & a_{r+1} \\
                    & b_1, & \ldots, & b_r
                 \end{array}\mid \lambda
\right):=\sum_{k=0}^{\infty}\frac{(a_1)_k\cdots(a_{r+1})_k}
{(b_1)_k\cdots(b_r)_k}\cdot\frac{\lambda^k}{k!}.
\end{align*}
Let $\widehat{\mathbb{F}_p^\times}$ be the group of all multiplicative characters of $\mathbb{F}_p^{\times}$. Let $\overline{\chi}$ denote the inverse of a multiplicative character $\chi$. We extend the domain of each $\chi\in\widehat{\mathbb{F}_p^\times}$ to $\mathbb{F}_p$ by simply setting $\chi(0):=0$ including the trivial character $\varepsilon.$ Let $\varphi$ be the quadratic character of $\mathbb{F}_p$. 
For multiplicative characters $\chi$ and $\psi$ of $\mathbb{F}_p$ the Jacobi sum is defined by
\begin{align}
J(\chi,\psi):=\sum_{y\in\mathbb{F}_p}\chi(y)\psi(1-y),\notag
\end{align}
and the normalized Jacobi sum known as binomial is defined by
$$
{\chi\choose \psi}:=\frac{\psi(-1)}{p}J(\chi,\overline{\psi}).
$$
Let $\mathbb{Z}_p$ denote the ring of $p$-adic integers and $\mathbb{Q}_p$ denote the field of $p$-adic numbers.
McCarthy \cite{mccarthy-pacific} developed a study of functions defined in terms of $p$-adic gamma functions with parameters in $\mathbb{Q}\cap\mathbb{Z}_p$ extending Greene's \cite{greene} hypergeometric functions over finite fields. Let $\Gamma_p(\cdot)$ denote the Morita's $p$-adic gamma function. Let $\omega$ denote the Teichm\"{u}ller character of $\mathbb{F}_p,$ satisfying $\omega(a)\equiv a\pmod{p},$ and $\overline{\omega}$ denote the character inverse of $\omega$. For $x\in\mathbb{Q}$ let $\lfloor x\rfloor$ denote the greatest integer less than or equal to $x$ and $\langle x\rangle$ denote the fractional part of $x$, satisfying $0\leq\langle x\rangle<1$.
We now recall the McCarthy's hypergeometric function in the $p$-adic setting.
\begin{definition}\cite[Definition 5.1]{mccarthy-pacific} \label{defin1}
Let $p$ be an odd prime and $t \in \mathbb{F}_p$.
For positive integer $n$ and $1\leq k\leq n$, let $a_k$, $b_k$ $\in \mathbb{Q}\cap \mathbb{Z}_p$.
Then 
\begin{align}
&{_n\mathbb{G}_n}\left[\begin{array}{cccc}
             a_1, & a_2, & \ldots, & a_n \\
             b_1, & b_2, & \ldots, & b_n
           \end{array}\mid t
 \right]_p:=\frac{-1}{p-1}\sum_{a=0}^{p-2}(-1)^{an}~~\overline{\omega}^a(t)\notag\\
&\times \prod\limits_{k=1}^n(-p)^{-\lfloor \langle a_k \rangle-\frac{a}{p-1} \rfloor -\lfloor\langle -b_k \rangle +\frac{a}{p-1}\rfloor}
 \frac{\Gamma_p(\langle a_k-\frac{a}{p-1}\rangle)}{\Gamma_p(\langle a_k \rangle)}
 \frac{\Gamma_p(\langle -b_k+\frac{a}{p-1} \rangle)}{\Gamma_p(\langle -b_k \rangle)}.\notag
\end{align}
\end{definition}
\noindent This function is called $p$-adic hypergeometric function. %Note that the values of the $_nG_n[\cdots]$ function depend only on the fractional part of the parameters $a_k$ and $b_k$. Hence, we may assume that $0\leq a_k,b_k<1$. 
%It is important to recall that there are two kinds of problems related to hypergeometric functions  that have been studied throughout the years. For the one kind the upper and lower case parameters appear in the hypergeometric function are varying and the evaluation point is fixed. 
Evaluation of particular values of special functions is one of the oldest problem in mathematics. For example, the follwing classical result of Gauss \cite{gauss} expresses the $_2F_1$ classical hypergeometric series in terms of quotients of classical gamma function. 
If $R(c-a-b)>0$ then
\begin{align}\label{gauss-value}
{_2F_1}\left(\begin{array}{cc}
           a, & b \\
           & c
         \end{array}\mid1\right)=\frac{\Gamma(c)\Gamma(c-a-b)}{\Gamma(c-a)\Gamma(c-b)}. 
\end{align}
The Gauss' theorem has been treated by N. Koblitz \cite{koblitz} and J. Diamond \cite{diamond} for classical hypergeometric series with $p$-adic variables, and also for ratios of such hypergeometric series. Ratios of hypergeometric series with $p$-adic variables were first considered by B. Dwork while studying the zeta function of a hypersurface. Similar to Gauss' theorem Kummer \cite{kummer}, Whipple \cite[p. 54]{lidl}, Saalch\"{u}tz \cite[p. 49]{slater}, Dixon \cite[p. 51]{slater}, and Watson \cite[p. 54]{slater} also evaluated  special values of classical hypergeometric functions. 
This paper focuses to study the values of the McCarthy's hypergeometric functions in the $p$-adic setting. For positive integer $n\geq1$ consider the following hypergeometric function in the $p$-adic setting
$${_{3n-1}G_{3n-1}}(p,t):={_{3n-1}\mathbb{G}_{3n-1}}\left[\begin{array}{ccccc}
\frac{1}{3n}, & \frac{2}{3n}, & \frac{3}{3n}, &\ldots, & \frac{3n-1}{3n}\vspace{1mm}\\
0,& \frac{1}{2}, & \frac{1}{3n-2}, & \ldots, & \frac{3n-3}{3n-2}
\end{array}\mid t\right]_p.$$
The next theorem presents the values of the function ${_{3n-1}G_{3n-1}}(p,t)$ in terms of the number of zeros of certain polynomial over $\mathbb{F}_p$.  Let $\delta: \widehat{\mathbb{F}_p^\times}\rightarrow\{0,1\}$ be defined by
$$
\delta(\chi):=\left\{
   \begin{array}{ll}
    1 , & \hbox{if $\chi=\varepsilon$;} \\
  0, & \hbox{if $\chi\neq\varepsilon$.}
   \end{array}
 \right.
$$
%%%%%%%%%%%%%%%%%%%%
%%%%%%%%%%%%%%%%%%%%%%
\begin{theorem}\label{general-1}
Let $n\geq1$ be a positive integer and $p\nmid3n(3n-2)$ be an odd prime. For $t\in\mathbb{F}_p^{\times}$ let
$\alpha=\frac{(-1)^{n}(3n-2)^{3n-2}}{(3n)^{3n}t}$ and $f_t(y):=y^{3n}-2y^{3n-1}+y^{3n-2}-(-1)^n4\alpha$ be a polynomial over $\mathbb{F}_p$. Then we have that
\begin{align}
{_{3n-1}G_{3n-1}}(p,t)=\left\{
   \begin{array}{ll}
    -1+\frac{1-p}{p}\varphi(\alpha)\delta(\varphi^n), & \hbox{if $f_t(y)$ has no root } \\ 
    &\hbox{in $\mathbb{F}_p$;} \\
 r -1+\frac{1-p}{p}\varphi(\alpha)\delta(\varphi^n) , & \hbox{if $f_t(y)$ has $r$ distinct}\\
 & \hbox{ roots in $\mathbb{F}_p$.}
   \end{array}
 \right.\notag
\end{align}
\end{theorem}
\noindent This theorem has some consequences. For example, 
if we put $n=1$ in Theorem \ref{general-1} then we have  
${_2G_2}(p,t)={_2\mathbb{G}_2}\left[\begin{array}{cc}
           \frac{1}{3}, & \frac{2}{3}\vspace{1mm} \\
          0, & \frac{1}{2}
         \end{array}\mid t\right]_p$ and obtain the following corollary. 
         
Let $S(t):=\left\{y\in\mathbb{F}_p:27y^3-27y^2+\frac{4}{t}\equiv0\pmod{p}\right\}$. For brevity we use the notation $\#S(t)$ to denote the number of elements of $S(t)$.
\begin{cor}\label{SV-2}
Let $p>3$ be a prime. Then we have 
\begin{align}\label{value-6}
{_2G_2}(p,1)=1, 
\end{align}
and if $1\neq t\in\mathbb{F}_p^\times$ then 
${_2G_2}(p,t)=\#S(t)-1.$
\end{cor}
\begin{remark}
If we put $a=\frac{1}{3},$ $b=\frac{2}{3}$ and  $c=1+\frac{1}{2}$ in \eqref{gauss-value} then we have
\begin{align}\label{gauss-3}
{_2F}_{1}\left(\begin{array}{cc}
           \frac{1}{3}, & \frac{2}{3} \\
           & 1+\frac{1}{2}
         \end{array}\mid1\right)=
         \frac{\Gamma(1+\frac{1}{2})\Gamma(\frac{1}{2})}{\Gamma(1+\frac{1}{6})\Gamma(\frac{5}{6})}=\frac{3}{2}.
         \end{align}
\eqref{value-6} can be described as $p$-adic analogue of \eqref{gauss-3}. This gives a motivation to represent Corollary \ref{SV-2} as some kind of extension of \eqref{gauss-3} in the $p$-adic setting.
\end{remark}
\noindent It is of interest to investigate the zeros of hypergeometric functions. The zero location of special classes of classical hypergeometric functions can be found in \cite{4, 6}.
By Theorem \ref{general-1} we obtain
\begin{cor}\label{zero-1}
If $p\nmid3n(3n-2)$ is an odd prime, then for $t\in\mathbb{F}_p^\times$ the following are true.
\begin{enumerate}
\item If $n$ is even then
${_{3n-1}G_{3n-1}}(p, t)\neq0$.
\item If $n$ is odd then ${_{3n-1}G_{3n-1}}(p,t)=0$
if and only if $y^{3n}-2y^{3n-1}+y^{3n-2}+\frac{4(3n-2)^{3n-2}}{(3n)^{3n}t}\equiv0\pmod{p}$ has exactly one incongruent solution modulo $p$.
\end{enumerate}
\end{cor}
 Let ${_2\widetilde{G}_2}(p, t)={_2\mathbb{G}_2}\left[\begin{array}{cc}
\frac{1}{6}, & \frac{5}{6}\vspace{1mm}\\
0, & \frac{1}{2}
\end{array}\mid t\right]_p$. We now examine the values of the function ${_2\widetilde{G}_2}(p,t)$.
\begin{theorem}\label{Special-value-1}
Let $p>3$ be a prime and $t\in\mathbb{F}_p^\times.$ Then we have 
\begin{align}\label{value-100}
{_2\widetilde{G}_2}(p, 1)=\varphi(3).
\end{align}
Moreover, if  $t\neq1$ then
${_2\widetilde{G}_2}(p, t)
=\varphi(3t)\times(\#S(t)-1).$
\end{theorem}

This theorem provides the possible range of the function ${_2\widetilde{G}_2}(p, t).$
\begin{cor}
Let $p>3$ be a prime and $t\in\mathbb{F}_p$. Then the possible values of the function ${_2\widetilde{G}_2}(p, t)$ can take are $0$, $\pm1$, and $\pm2$.
\end{cor}
\noindent Another important topic to be discussed is that of transformation identities. 
Classical hypergeometric series satisfy many transformation identities. Kummer \cite{kummer-2, kummer-3} gave a list of 24 solutions to the differential equation for the hypergeometric series
$$x(l- x)y''+(c-(a+b+l)x)y'- aby=0.$$ Any three such solutions are linearly dependent. The resulting linear combinations yield transformations for classical hypergeometric series. For example, Kummer \cite[p. 4 eq. (1)]{bailey} obtained the following linear transformation.
\begin{align}\label{kummer-transformation}
&{_2F_1}\left(\begin{array}{cc}
           a, & b \\
           & c
         \end{array}\mid z\right)=\frac{\Gamma(c)\Gamma(c-a-b)}{\Gamma(c-a)\Gamma(c-b)}\cdot
         {_2F_1}\left(\begin{array}{cc}
           a, & b \\
           & 1+a+b-c
         \end{array}\mid1-z\right)\\
         &+\frac{\Gamma(c)\Gamma(a+b-c)}{\Gamma(a)\Gamma(b)}(1-z)^{c-a-b}\cdot
         {_2F_1}\left(\begin{array}{cc}
           c-a, & c-b \\
           & 1+c-a-b
         \end{array}\mid1-z\right).\notag
\end{align} 
Euler \cite[p. 10]{slater}, Whipple \cite{whipple}, Dixon \cite{dixon} also studied transformation properties of classical hypergeometric series. It is often beneficial to study such results in different settings. For example, Fuselier-McCarthy \cite{fm} established certain transformation identities for hypergeometric series in the $p$-adic setting. Interestingly, they proved a transformation of $p$-adic hypergeometric series analogous to a particular case of Whipple's transformation result for ${_3F_2}$-classical hypergeometric series and using such transformations they proved one supercongruence conjecture of Rodriguez-Villegas between a truncated 
${_4F_3}$-classical hypergeometric series and the Fourier coefficients of a certain weight four modular form. These consequences attract our interest to further study other analogues of classical hypergeometric transformations in the $p$-adic setting. In this paper, we aim to establish a particular case of Kummer's identity \eqref{kummer-transformation}.
It is important to recall that Kummer's methods of deriving solutions of the differential equation involved changes of variables that do not affect the form of the differential equation. In our case we do not have any differential equation so we must develop different techniques to settle analogous transformations for the hypergeometric series in the $p$-adic setting.
We obtain a $p$-adic analogue of Kummer's identity \eqref{kummer-transformation} for the functions 
${_2G_2}(p,t)$ and ${_2\widetilde{G}_2}(p,t)$ in the next theorem as an application of Corollary \ref{SV-2} and Theorem \ref{Special-value-1}.
\begin{theorem}\label{kummer-transformation-1}
Let $p>3$ be a prime and $1\neq t\in\mathbb{F}_p^\times$. For fixed $t$ let $f_t(y)=27y^3-27y^2+\frac{4}{t}$, and 
$g_t(y)=27y^3-27y^2+\frac{4}{1-t}$ be two polynomials over $\mathbb{F}_p$. If $f_t(y)\equiv0\pmod{p}$ has $r_1$ incongruent solutions and $g_t(y)\equiv0\pmod{p}$ has $r_2$ incongruent solutions modulo $p$ then 
\begin{align}
{_2G_2}(p,t)&={_2G_2}(p,1-t)+{_2\widetilde{G}_2}(p,1-t)+(r_1-r_2)+(1-r_2)\varphi(3(1-t)).\notag
\end{align}
%%%%%%%%%%%%%%%%%%%%%%%%%%%%%%%%%%%%%%%%%%%%%%%%%%%%%

\end{theorem}
%%%%%%%%%%%%%%%%%%%%%%%%%%%%%%%%%%%%%%%%%%
\noindent The rest part of the paper is organized as follows. Section 2 contains some basic definitions including Gauss sum, and $p$-adic gamma function. We also present some preliminary results along with some important theorems including Hasse-Devanport Theorem and Gross-Koblitz formula in Section 2. In Section 3 we prove two propositions which play  important roles in proving the main results. We give the proofs of the results in Section 3.
\section{Notation and Preliminary results}
\subsection{Multiplicative characters and Gauss sums:}
The following result is the orthogonality relation of multiplicative characters.
\begin{lemma}\cite[Chapter 8]{ireland}
Let $p$ be an odd prime. Then we have
\begin{align}\label{orthogonal-1}
\sum_{\chi\in\widehat{\mathbb{F}_p^\times}}\chi(x)=\left\{
   \begin{array}{ll}
    p-1 , & \hbox{if $x=1$;} \\
  0, & \hbox{if $x\neq1$.}
   \end{array}
 \right.
\end{align}
\end{lemma}
\noindent Let $\zeta_p$ denote a fixed primitive $p$-th root of unity. Then for multiplicative character
 $\chi$ of $\mathbb{F}_p$ the Gauss sum is defined by
\begin{align}
g(\chi):=\sum\limits_{x\in \mathbb{F}_p}\chi(x)~\zeta_p^x.\notag
\end{align}
It is easy to verify that $g(\varepsilon)=-1$. 
The following product of Gauss sums is very useful in the proofs of our results.
\begin{lemma}\cite[eq. 1.12]{greene} Let $\chi\in \widehat{\mathbb{F}_p^\times}.$ Then
\begin{align}\label{inverse}
g(\chi)g(\overline{\chi})=p\chi(-1)-(p-1)\delta(\chi).
\end{align}
\end{lemma}
\noindent The Hasse-Davenport formula is also very important product formula of Gauss sums.
\begin{theorem}\cite[Hasse-Davenport relation, Theorem 11.3.5]{berndt}
Let $\chi$ be a multiplicative character of order $m$ of $\mathbb{F}_p$ for some positive integer $m.$ For a multiplicative character $\psi$ of $\mathbb{F}_p$ we  have
\begin{align}\label{dh}
\prod_{i=0}^{m-1}g(\psi\chi^i)=g(\psi^m)\psi^{-m}(m)\prod_{i=1}^{m-1}g(\chi^i).
\end{align}
\end{theorem}
\noindent The next lemma gives a relation between Gauss and Jacobi sums.
\begin{lemma}\cite[eq. 1.14]{greene} 
Let $\chi_1,\chi_2\in \widehat{\mathbb{F}_p^\times}.$ Then
\begin{align}\label{gauss-jacobi}
J(\chi_1,\chi_2)=\frac{g(\chi_1)g(\chi_2)}{g(\chi_1\chi_2)}+(p-1)\chi_2(-1)\delta(\chi_1\chi_2).
\end{align}
\end{lemma}
\noindent Let $\chi,\psi$ be multiplicative characters of $\mathbb{F}_p.$ Then from \cite[eq. 2.12, eq. 2.7]{greene} we have
\begin{align}\label{rel-1}
&{\chi\choose\varepsilon}={\chi\choose\chi}=-\frac{1}{p}+\frac{p-1}{p}\delta(\chi),\\
&\label{rel-2}
{\chi\choose\psi}={\psi\overline{\chi}\choose\psi}\psi(-1).
\end{align}

\subsection{$p$-adic preliminaries:}
Let $\overline{\mathbb{Q}_p}$ denote the algebraic closure of $\mathbb{Q}_p$ and $\mathbb{C}_p$ denote the completion of $\overline{\mathbb{Q}_p}$. 
We now recall the $p$-adic gamma function, for further details see \cite{kob}.
For a positive integer $n$
the $p$-adic gamma function $\Gamma_p(n)$ is defined as
\begin{align}
\Gamma_p(n):=(-1)^n\prod\limits_{0<j<n,p\nmid j}j.\notag
\end{align}
 It can be extended to all $x\in\mathbb{Z}_p$ by setting $\Gamma_p(0):=1$ and for $x\neq0$
\begin{align}
\Gamma_p(x):=\lim_{x_n\rightarrow x}\Gamma_p(x_n).\notag
\end{align}
%If $x\in\mathbb{Z}_p$ then the $p$-adic analogue of Euler's reflection formula is 
%\begin{align}\label{prod-3}
%\Gamma_p(x)\Gamma_p(1-x)=(-1)^{a_0(x)},
%\end{align}
%where $a_0(x)\equiv x\pmod{p}$ such that $a_0(x)\in\{1,2,\ldots,p\}$, for more details we refer \cite{gross}. 
An important product formula of $p$-adic gamma functions is given below.
If $m\in\mathbb{Z}^+,$ 
$p\nmid m$  and $x=\frac{r}{p-1}$ with $0\leq r\leq p-1$ then
\begin{align}\label{prod-1}
\prod_{h=0}^{m-1}\Gamma_p\left(\frac{x+h}{m}\right)=\omega(m^{(1-x)(1-p)})~\Gamma_p(x)\prod_{h=1}^{m-1}\Gamma_p\left(\frac{h}{m}\right).
\end{align}
From \cite{mccarthy-pacific} we also state that if $t\in\mathbb{Z}^{+}$ and $p\nmid t$ then for $0\leq j\leq p-2$ we have
\begin{align}\label{new-prod-1}
\omega(t^{tj})\Gamma_p\left(\left\langle\frac{tj}{p-1}\right\rangle\right)\prod_{h=1}^{t-1}\Gamma_p\left(\frac{h}{t}\right)
=\prod_{h=0}^{t-1}\Gamma_p\left(\left\langle\frac{h}{t}+\frac{j}{p-1}\right\rangle\right),
\end{align}
and
\begin{align}\label{prod-2}
\omega(t^{-tj})\Gamma_p\left(\left\langle\frac{-tj}{p-1}\right\rangle\right)\prod_{h=1}^{t-1}\Gamma_p\left(\frac{h}{t}\right)
=\prod_{h=1}^{t}\Gamma_p\left(\left\langle\frac{h}{t}-\frac{j}{p-1}\right\rangle\right).
\end{align}
Let $\pi \in \mathbb{C}_p$ be the fixed root of the polynomial $x^{p-1} + p$, which satisfies the congruence condition
$\pi \equiv \zeta_p-1 \pmod{(\zeta_p-1)^2}$. Then the following result is known as the Gross-Koblitz formula. 
\begin{theorem}\cite[Gross-Koblitz formula]{gross}\label{gross-koblitz} For $j\in \mathbb{Z}$,
\begin{align}
g(\overline{\omega}^j)=-\pi^{(p-1)\langle\frac{j}{p-1} \rangle}\Gamma_p\left(\left\langle \frac{j}{p-1} \right\rangle\right).\notag
\end{align}
\end{theorem}
\noindent We now state a lemma that is obtained by applying Gross-Koblitz formula.
\begin{lemma}
Let $1\leq j\leq p-2.$ Then
\begin{align}\label{lemma-1}
\Gamma_p\left(\left\langle1-\frac{j}{p-1}\right\rangle\right)\Gamma_p\left(\left\langle\frac{j}{p-1}\right\rangle\right)
=-\omega^j(-1).
\end{align}
\end{lemma}
\begin{proof}
We first apply Gross-Koblitz formula (Theorem \ref{gross-koblitz}) on the left hand side of \eqref{lemma-1} and then apply \eqref{inverse} to obtain the result.
\end{proof}
%\begin{lemma}
%For $1\leq j\leq p-2$ we have
%\begin{align}\label{lemma-2}
%&\frac{(-p)^{-\lfloor\frac{1}{2}+\frac{j}{p-1}\rfloor}}{\Gamma_p(\frac{1}{2})}~
%\Gamma_p\left(\left\langle\frac{1}{2}+\frac{j}{p-1}\right\rangle\right)
%\Gamma_p\left(\left\langle1-\frac{j}{p-1}\right\rangle\right)\notag\\
%&=\frac{1}{p}\sum_{t\in\mathbb{F}_p^\times}
%\overline{\omega}^j(-t)\varphi(t(t-1)).
%\end{align}
%\end{lemma}
%\begin{proof}
%Let $U=\frac{(-p)^{-\lfloor\frac{1}{2}+\frac{j}{p-1}\rfloor}}{\Gamma_p(\frac{1}{2})}~
%\Gamma_p\left(\left\langle\frac{1}{2}+\frac{j}{p-1}\right\rangle\right)
%\Gamma_p\left(\left\langle1-\frac{j}{p-1}\right\rangle\right).$ If we apply Gross-Koblitz formula (Theorem \ref{gross-koblitz}), \eqref{inverse}, and \eqref{gauss-jacobi} then we obtain
% \begin{align*}
% U&=\frac{\varphi\omega^j(-1)}{p}J(\varphi\overline{\omega}^j,\varphi)
% =\frac{1}{p}\sum_{t\in\mathbb{F}_p^\times}\overline{\omega}^j(-t)\varphi(t(t-1)).
% \end{align*}
% This completes the proof of the lemma.
% \end{proof}
\section{Proof of the theorems}
\noindent We begin this section with two propositions. These proposition are used to prove Theorem \ref{general-1}, and Theorem
 \ref{Special-value-1}.
\begin{proposition}\label{new-prop-1}
Let $n\geq1$ be a positive integer and $p\nmid3n(3n-2)$ be an odd prime. 
For $x\in\mathbb{F}_p^{\times}$ let $\alpha=\frac{(-1)^{n}(3n-2)^{3n-2}}{(3n)^{3n}x}$, and consider
$f_{x}(y)=y^{3n}-2y^{3n-1}+y^{3n-2}-(-1)^n4\alpha\in\mathbb{F}_p[y].$\\  If $C(n,x)=\displaystyle\sum_{\chi\in\widehat{\mathbb{F}_p^\times}}g(\chi^{3n})g(\varphi\overline{\chi})g(\overline{\chi})
g(\overline{\chi}^{3n-2})
~\chi(\alpha)$ then we have
\begin{align*}
&C(n,x)+(p-1)^2g(\varphi)(1+\varphi(\alpha)\delta(\varphi^n))\\
&=\left\{
   \begin{array}{ll}
    0 , & \hbox{if $f_x(y)\equiv0\pmod{p}$ has no solution modulo $p$;} \\
  rp(p-1)g(\varphi), & \hbox{if $f_x(y)\equiv0\pmod{p}$ has $r$ incongruent }\\
  ~~& \hbox{\hspace{2.2cm} solutions modulo $p$.}
   \end{array}
 \right.
\end{align*}
\end{proposition}
\begin{proof}
Multiplying both the numerator and denominator by $g(\chi^2)$ we can write
\begin{align}\label{eqn-18}
C(n,x)=\sum_{\chi\in\widehat{\mathbb{F}_p^\times}}\frac{g(\chi^{3n})g(\overline{\chi}^{3n-2})}{g(\chi^2)}~g(\chi^2)
g(\varphi\overline{\chi})g(\overline{\chi})~\chi(\alpha).
\end{align}
Applying Hasse-Davenport \eqref{dh} with $m=2$ we have
\begin{align}\label{eqn-19}
g(\varphi\overline{\chi})g(\overline{\chi})=g(\varphi)\overline{\chi}(2^{-2})g(\overline{\chi}^2).
\end{align}
Substituting \eqref{eqn-19} into \eqref{eqn-18} and then applying \eqref{inverse}, \eqref{gauss-jacobi} we have
\begin{align}\label{eqn-20}
\frac{C(n,x)}{g(\varphi)}&=p\sum_{\chi\in\widehat{\mathbb{F}_p^{\times}}} J(\chi^{3n}, \overline{\chi}^{3n-2})\chi(4\alpha)
-(p-1)^2(1+\varphi(\alpha)\delta(\varphi^n)).
\end{align}
 \eqref{rel-2} yields $ J(\chi^{3n}, \overline{\chi}^{3n-2})=\chi^{n}(-1)J(\overline{\chi}^2,\overline{\chi}^{3n-2})$.
 Using this, and replacing $\chi$ by $\overline{\chi}$ in \eqref{eqn-20} we obtain
\begin{align}
C(n,x)&=pg(\varphi)\sum_{\chi\in\widehat{\mathbb{F}_p^\times},~ y\in\mathbb{F}_p}
\chi\left(\frac{y^{3n-2}(1-2y+y^2)}{(-1)^n4\alpha}\right)\notag\\
&-(p-1)^2g(\varphi)(1+\varphi(\alpha)\delta(\varphi^n)).\notag
\end{align}
Orthogonality relation of multiplicative characters gives that the double sum is non zero if and only if the congruence $f_x(y)\equiv0\pmod{p}$ admits a solution. Considering the number of incongruent solutions of $f_x(y)\equiv0\pmod{p}$ we 
obtain the required identity.
\end{proof}

%%%%%%%%%%%%%%%%%%%%%%%%%%%%%%%%%%%%%%%%%%
\begin{proposition}\label{new-prop-2}
Let $n\geq1$ be a positive integer and $p\nmid3n(3n-2)$ be a prime. For $x\in\mathbb{F}_p^{\times}$ let 
$\alpha=\frac{(-1)^n(3n-2)^{3n-2}}{(3n)^{3n}x}$, and\\
Let $C(n,x)=\displaystyle\sum_{\chi\in\widehat{\mathbb{F}_p^\times}}g(\chi^{3n})g(\varphi\overline{\chi})g(\overline{\chi})
g(\overline{\chi}^{3n-2})
~\chi(\alpha)$.
Then we have
\begin{align*}
C(n,x)=(p-1)g(\varphi)(1+p\cdot{_{3n-1}G_{3n-1}}(p, x)).
\end{align*}
\end{proposition}
\begin{proof}
 Replacing $\chi$ by $\omega^j$ and then applying Gross-Koblitz formula we obtain
\begin{align}\label{eqn-25}
C(n,x)&=\sum_{j=0}^{p-2}\omega^j(\alpha)\pi^{(p-1)\ell_j}~
\Gamma_p\left(\left\langle\frac{-3nj}{p-1}\right\rangle\right)\Gamma_p\left(\frac{j}{p-1}\right)\\
&\times\Gamma_p\left(\left\langle\frac{1}{2}+\frac{j}{p-1}\right\rangle\right)\Gamma_p\left(\left\langle\frac{(3n-2)j}{p-1}\right\rangle\right)\notag,
\end{align}
where
$\ell_j=\frac{1}{2}-\left\lfloor\frac{-3nj}{p-1}\right\rfloor-\left\lfloor\frac{1}{2}+\frac{j}{p-1}\right\rfloor-\left\lfloor\frac{(3n-2)j}{p-1}\right\rfloor$.
Simplifying the sum by applying \eqref{prod-2} (with $t=3n$), and \eqref{new-prod-1} (with $t=3n-2$) and then using \eqref{lemma-1} we bring down the sum to
\begin{align}
C(n,x)&=\pi^{\frac{(p-1)}{2}}\Gamma_p\left(\frac{1}{2}\right)-
\sum_{j=1}^{p-2}\overline{\omega}^j((-1)^{n+1}x)\pi^{(p-1)\ell_j}~\Gamma_p\left(\left\langle\frac{1}{2}+\frac{j}{p-1}\right\rangle\right)\notag\\
&\times
\prod_{h=1}^{3n-1}\frac{\Gamma_p\left(\left\langle\frac{h}{3n}-\frac{j}{p-1}\right\rangle\right)}{\Gamma_p(\frac{h}{3n})}
\prod_{h=0}^{3(n-1)}\frac{\Gamma_p\left(\left\langle\frac{h}{3n-2}+\frac{j}{p-1}\right\rangle\right)}{\Gamma_p(\frac{h}{3n-2})}.\notag
\end{align}
For $1\leq j\leq p-2$ consider $\left\lfloor\frac{-3nj}{p-1}\right\rfloor=-3n+s$ for some $s\in\mathbb{Z}$ with $0\leq s\leq3n-1$ we obtain $\left\lfloor\dfrac{-3nj}{p-1}\right\rfloor=-1+
\displaystyle\sum_{h=1}^{3n-1}\left\lfloor\frac{h}{3n}-\frac{j}{p-1}\right\rfloor$. Also, let
$\left\lfloor\frac{(3n-2)j}{p-1}\right\rfloor=k$ for some $k\in\mathbb{Z}$ with $0\leq k\leq3n-3$. Then it is easy to verify that $\left\lfloor\dfrac{(3n-2)j}{p-1}\right\rfloor=
\displaystyle\sum_{h=1}^{3n-3}\left\lfloor\frac{h}{3n-2}+\frac{j}{p-1}\right\rfloor$.
Substituting these identities in the expression of $\ell_j$ and rearranging the terms we have
\begin{align}\label{10000}
C(n,x)&=(1-p)\pi^{\frac{(p-1)}{2}}\Gamma_p(1/2)(1+p\cdot{_{3n-1}G_{3n-1}}(p, x)).
\end{align}
By Gross-Koblitz formula we have
$g(\varphi)=-\pi^{\frac{(p-1)}{2}}\Gamma_p(1/2)$. Putting this in \eqref{10000} we obtain the required identity.
This completes the proof of the proposition.
\end{proof}
\begin{proof}[Proof of Theorem \ref{general-1}]
Let $C(n,t)=\displaystyle\sum_{\chi\in\widehat{\mathbb{F}_p^\times}}g(\chi^{3n})g(\varphi\overline{\chi})g(\overline{\chi})
g(\overline{\chi}^{3n-2})
\chi(\alpha)$. Now, applying Proposition \ref{new-prop-1}, Proposition \ref{new-prop-2} and then comparing the terms we obtain the result. This completes the proof of the theorem.
\end{proof}

%%%%%%%%%%%%%%%%%%%%%%%%%%%%%%%%%%%%%%%%%%%
%%%%%%%%%%%%%%%%%%%%%%%%%%%%%%%%%%%%%%%%%%%%%
\begin{proof}[Proof of Corollary \ref{SV-2}] If we put $n=1$ in Theorem \ref{general-1} then we have $
{_2G_2}(p,t)=r-1$, 
where $0\leq r\leq3$ is the number of incongruent solutions of $27y^3-54y^2+27y-\frac{4}{t}\equiv0\pmod{p}$ modulo $p$. Replacing $y$ by $(1-y)$ we obtain $r=\#S(t)$. In particular, if we put $t=1$ then $\frac{1}{3}$, and $\frac{2}{3}$ are the only incongruent solutions of $27y^3-27y^2+4\equiv0\pmod{p}$ modulo $p$. Hence, ${_2G_2}(p,1)=1$.
\end{proof}
\begin{proof}[Proof of Corollary \ref{zero-1}]
Let $\alpha=\frac{(-1)^{n}(3n-2)^{3n-2}}{(3n)^{3n}t}$. Let $r$ be the number of incongruent solutions of $y^{3n}-2y^{3n-1}+y^{3n-2}-(-1)^n4\alpha\equiv\pmod{p}$ modulo $p$. If $n$ is even then by Theorem \ref{general-1} we have
\begin{align*}
{_{3n-1}G_{3n-1}}(p, t)= r -1+\frac{1-p}{p}\varphi(t)\neq0.
\end{align*}
Again, if $n$ is odd then by Theorem \ref{general-1} we have
\begin{align}\label{new-eq-105}
{_{3n-1}G_{3n-1}}(p, t)= r -1.
\end{align}
\eqref{new-eq-105} will be zero only if $r=1$. 
\end{proof}
\begin{proof}[Proof of Theorem \ref{Special-value-1}]
Let $B_{t}=\displaystyle\sum_{\chi\in\widehat{\mathbb{F}_p^\times}}g(\chi^3)g(\varphi\overline{\chi})g(\overline{\chi})^2
~\overline{\chi}\left(-27t\right)$.
From Proposition \ref{new-prop-2} we can write $B_t=C(1,t)$. Then \eqref{10000} yields
\begin{align}\label{eq-10001}
B_{t}&= (1-p)\pi^{\frac{(p-1)}{2}}\Gamma_p(1/2)\left(1+p\cdot~{_{2}G_{2}}(p, t)\right)\notag\\
&=p\pi^{\frac{(p-1)}{2}}
\sum_{j=0}^{p-2}(-p)^{(-\lfloor\frac{1}{3}-\frac{j}{p-1}\rfloor-
\lfloor\frac{2}{3}-\frac{j}{p-1}\rfloor-\lfloor\frac{1}{2}+\frac{j}{p-1}\rfloor-\lfloor\frac{j}{p-1}\rfloor)}\\
&\times \overline{\omega}^j(t)\frac{\Gamma_p\left(\left\langle\frac{1}{3}-\frac{j}{p-1}\right\rangle\right)
\Gamma_p\left(\left\langle\frac{2}{3}-\frac{j}{p-1}\right\rangle\right)}{\Gamma_p(\frac{1}{3})\Gamma_p(\frac{2}{3})}
\Gamma_p\left(\frac{j}{p-1}\right)\notag\\
&\times\Gamma_p\left(\left\langle\frac{1}{2}+\frac{j}{p-1}\right\rangle\right)+(1-p)\pi^{\frac{(p-1)}{2}}\Gamma_p\left(\frac{1}{2}\right).\notag
\end{align}
Now, replacing $j$ by $j-\frac{p-1}{2}$ in \eqref{eq-10001} we obtain
\begin{align}
B_{t}&=\varphi(t)p\pi^{\frac{(p-1)}{2}}
\sum_{j=0}^{p-2}\overline{\omega}^j(t)(-p)^{(-\lfloor\frac{1}{6}-\frac{j}{p-1}\rfloor-
\lfloor\frac{5}{6}-\frac{j}{p-1}\rfloor-\lfloor\frac{1}{2}+\frac{j}{p-1}\rfloor-\lfloor\frac{j}{p-1}\rfloor)}\notag\\
&\times\frac{\Gamma_p\left(\left\langle\frac{1}{6}-\frac{j}{p-1}\right\rangle\right)
\Gamma_p\left(\left\langle\frac{5}{6}-\frac{j}{p-1}\right\rangle\right)}{\Gamma_p(\frac{1}{3})\Gamma_p(\frac{2}{3})}
\Gamma_p\left(\frac{j}{p-1}\right)\Gamma_p\left(\left\langle\frac{1}{2}+\frac{j}{p-1}\right\rangle\right)\notag\\
&+(1-p)\pi^{\frac{(p-1)}{2}}\Gamma_p(1/2)\notag\\
&=(1-p)\pi^{\frac{(p-1)}{2}}\Gamma_p(1/2)\left(1+p\varphi(t)
\frac{\Gamma_p(\frac{1}{6})\Gamma_p(\frac{5}{6})}{\Gamma_p(\frac{1}{3})\Gamma_p(\frac{2}{3})}
\cdot {_2\widetilde{G}_2}(p, t)\right).\notag
\end{align}
\eqref{new-prod-1} gives
$\frac{\Gamma_p(\frac{1}{6})\Gamma_p(\frac{5}{6})}{\Gamma_p(\frac{1}{3})\Gamma_p(\frac{2}{3})}=\varphi(3)$ and Gross-Koblitz gives $g(\varphi)=-\pi^{\frac{p-1}{2}}\Gamma_p(1/2)$. Substituting these two in the above sum we have
\begin{align}\label{eq-10002}
B_{t}&=(p-1)g(\varphi)\left(1+ p\varphi(3t)
~{_2\widetilde{G}_2}(p, t)\right).
\end{align}
Also, by Proposition \ref{new-prop-2} we have
\begin{align}\label{new-eq-10002}
B_{t}&=(p-1)g(\varphi)\left(1+ p
~{_2{G}_2}(p, t)\right).
\end{align}
Combining \eqref{eq-10002}, and \eqref{new-eq-10002} we can write
${_2{G}_2}(p, t)=\varphi(3t)\cdot{_2\widetilde{G}_2}(p, t)$. Now, by Corollary \ref{SV-2} we obtain the required identity.
 This completes the proof.
\end{proof}

%%%%%%%%%%%%%%%%%%%%%%%%%%%%%%%%%%%%%%%%%%%%
%%%%%%%%%%%%%%%%%%%%%%%%%%%%%%%%%%%%%%%%%%%%%

%%%%%%%%%%%%%%%%%%%%%%%%%%%%%%%%%%%%%%%%%%%%%
%%%%%%%%%%%%%%%%%%%%%%%%%%%%%%%%%%%%%%%%%%%%%
\begin{proof}[Proof of Theorem \ref{kummer-transformation-1}]
By Corollary \ref{SV-2} and Theorem \ref{Special-value-1}  we have
\begin{align}\label{eq-50001}
&{_2G_2}(p,t)=r_1-1\\\label{eq-500005}
&{_2G_2}(p,1-t)=r_2-1,\\\label{eq-50002}
&{_2\widetilde{G}_2}(p,1-t)=(r_2-1)\varphi(3(1-t)).
\end{align}
Combining \eqref{eq-50001},  \eqref{eq-500005}, and  \eqref{eq-50002} we readily obtain the transformation. 
This completes the proof of the theorem.
\end{proof}

%%%%%%%%%%%%%%%%%%%%%%%%%%%%%%%%%%%%%%%%%%%%%%%%%%
%%%%%%%%%%%%%%%%%%%%%%%%%%%%%%%%%%%%%%%%%%%%%%%%%%%%

\end{document}